\documentclass[11pt,reqno]{amsart}
\usepackage{amsthm}
\usepackage{mathtools}
\usepackage{amssymb}
\usepackage{hyperref}
\usepackage{enumitem}
\usepackage{graphicx}
\usepackage{comment}

\numberwithin{equation}{section}
\newtheorem{theorem}{Theorem}[section]

\newtheorem{proposition}[theorem]{Proposition}

\newtheorem{corollary}[theorem]{Corollary}

\theoremstyle{definition}
\newtheorem{remark}[theorem]{Remark}

\renewcommand{\epsilon}{\varepsilon}
\renewcommand{\Re}{\textnormal{Re}}

\title[Tube stability for DNLS]{The stability of degenerate solitons for derivative nonlinear Schr\"odinger equations}
\author{TaeGyu Kim}
\begin{document}
\maketitle

\begin{abstract}
	In this paper, we consider the following nonlinear Schr\"odinger equation with derivative:
	\begin{align*}
		i\partial_tu+\partial_{xx}u+i|u|^{2}\partial_xu+b|u|^4u=0, \quad (t,x) \in \mathbb{R}\times\mathbb{R}, \quad b\geq 0.
	\end{align*}
	For the case $b=0$, the original DNLS, Kwon and Wu \cite{KwonWu2018} proved the conditional orbital stability of degenerate solitons including scaling, phase rotation, and spatial translation with a non-smallness condition, $\|u(t)\|_{L^6}^6> \sqrt{\delta}$. In this paper, we remove this condition for the non-positive initial energy and momentum, and we extend the stability result for $b\geq0$.
\end{abstract}

\section{Introduction}
We consider the initial value problem 
\begin{align}\label{dnls}
\begin{cases}
i\partial_tu+\partial_{xx}u+i|u|^{2}\partial_xu+b|u|^4u=0, \quad (t,x) \in \mathbb{R}\times\mathbb{R}
\\
u(0)=u_0
\end{cases}
\end{align}
where $b\geq0$ and $u: \mathbb{R}\times \mathbb{R}\to \mathbb{C}$. It is well-known that \eqref{dnls} is locally well-posed in the energy space $H^1(\mathbb{R})$ (see \cite{Ozawa1996}). The equation \eqref{dnls} has the following three conserved quantities
\begin{align}
	E(u)&={1\over 2}\|\partial_x u\|_{L^2}^2-{1\over 4}(i|u|^{2}\partial_x u, u)_{L^2} -{b\over6}\|u\|_{L^6}^6, \tag{Energy}
	\\
	M(u)&=\|u\|_{L^2}^2, \tag{Mass}
	\\
	P(u)&=(i\partial_x u, u)_{L^2} \tag{Mommentum}
\end{align}
where the inner product $(\cdot, \cdot)_{L^2}$ is defined by
\begin{align*}
	(f,g)_{L^2}=\Re \int_{\mathbb{R}}v(x)\overline{w(x)}dx.
\end{align*}
When $b=0$, the equation
\begin{align}\label{DNLSoriginal}
	i\partial_tu+\partial_{xx}u+i|u|^{2}\partial_xu=0, \quad (t,x) \in \mathbb{R}\times\mathbb{R} 
\end{align}
is known as a standard derivative nonlinear Schr\"odinger equation and sometimes referred to as the Chen-Lee-Liu equation \cite{ChenLeeLiu1979}. The equation \eqref{DNLSoriginal} is a gauge equivalent form of the well-known derivative nonlinear Schr\"odinger equation
\begin{align}\label{DNLSoriginGauge}
	i\partial_t\psi+\partial_{xx}\psi+i\partial_x(|\psi|^2\psi)=0, \quad (t,x) \in \mathbb{R}\times\mathbb{R} \tag{DNLS}
\end{align}
after applying the following gauge tranformation to the solution of \eqref{DNLSoriginal}
\begin{align*}
	\psi(t,x)=u(t,x)\exp\left(-\frac{i}{2}\int_{-\infty}^x|u(t,x)|^2dx\right).
\end{align*}
\eqref{DNLSoriginGauge} was introduced as a model in plasma physics \cite{Mio1975ModifiedNS,mjlhus_1976} and shown to be completely integrable \cite{KaupNewell1978}. We can consider the equation \eqref{dnls} as a generalization of \eqref{DNLSoriginGauge}.

The equation \eqref{dnls} has two parameter family of solitons
\begin{align*}
	u_{\omega,c}(t,x)=e^{i\omega t}\phi_{\omega,c}(x-ct)
\end{align*}
where $\phi_{\omega,c}$ is given by \eqref{2solitonForm}. For $b\geq0$, there exists a unique $\kappa_0=\kappa_0(b)\in(0,1]$ such that
\begin{align*}
	E(\phi_{\omega,2\kappa_0\sqrt{\omega}})=
	P(\phi_{\omega,2\kappa_0\sqrt{\omega}})=0,
\end{align*}
which means that the soliton $\phi_{\omega,2\kappa_0\sqrt{\omega}}$ corresponds to the degenerate case. If $b=0$, $\kappa_0(0)=1$ and this means that the algebraic solitons of \eqref{DNLSoriginal} are the degenerate case. We also note that $0<\kappa_0(b)<1$ for $b>0$, and the degenerate solitons have exponential decay at space infinity. In this paper, by using the variational characterization introduced in \cite{KwonWu2018} with the argument used in (\cite{Hayashi2021}, Theorem 1.7-(vi)), we establish the orbital stability of the degenerate solitons including scaling, phase rotation, and spatial translation for non-positive initial energy and momentum.

For $b=0$, Guo and Wu \cite{GuoWu1995} show orbital stability for $c<0$ and $\omega>c^2/4$. After this, Colin and Ohta \cite{ColinOhta2006} proved the orbital stability of solitary waves for $\omega>c^2/4$ using the variational method in \cite{Shatah1983}. For the end-point case $c=2\sqrt{\omega}$, the conditional result for orbital stability of solitary waves was proved by Kwon and Wu \cite{KwonWu2018}. Besides these, there are stability or instability results for DNLS with various $b$ in \cite{Ohta2014,NingOhtaWu2017,Ning2020,Hayashi2022negativeb,FukayaHayashi2022}, and for gDNLS in \cite{LiuSimpsonSulem2013,Guo2018,MiaoTangXu2023}, or references therein. 

As mentioned above, Kwon and Wu \cite{KwonWu2018} obtained conditional orbital stability (up to symmetries) of the degenerate soliton, in the sense that $|u(t)|_{L^6}$ did not have to be small when $b=0$. 
\begin{theorem}[\cite{KwonWu2018}]
	Let $b=0$ and $u(t)$ be the solution to the \eqref{DNLSoriginGauge} wtih $u(0)=u_0$. For any $\epsilon>0$, there exists a $\delta=\delta(\epsilon) $such that if
	\begin{align*}
		\|u_0-\phi_{1,2}\|_{H^1}\leq \delta,
	\end{align*}
	then for any $t\in I_{\text{max}}$ (the maximal lifespan), either
	\begin{align*}
		\|u(t)\|_{L^6}^6\leq \sqrt{\delta}
	\end{align*}
	or there exist $\theta=\theta(t)\in[0,2\pi)$, $y=y(t)\in\mathbb{R}$, and $\lambda=\lambda(t)\in [\lambda_0, \infty)$ for some constant $\lambda_0>0$, such that
	\begin{align*}
		\left\|u(t)-e^{i\theta}\lambda^{1/2}\phi_{1, 2}\left(\lambda(\ \cdot -y)\right)\right\|_{H^1}<\epsilon.
	\end{align*}
\end{theorem}
In our main theorem, we resolve the non-smallness condition, $\|u(t)\|_{L^6}^6> \sqrt{\delta}$, for non-positive initial energy and momentum.
\begin{theorem}\label{MainTheorem}
	Let $b\geq 0$ and $u(t)$ be the solution to the \eqref{dnls} wtih $u(0)=u_0$. For any $\epsilon>0$, there exists a $\delta=\delta(\epsilon)$ such that if $M(\phi_{1, 2\kappa_0})<M(u_0)<M(\phi_{1, 2\kappa_0})+\delta$, $E(u_0)\leq 0$, and $P(u_0)\leq 0$, then there exist $\theta=\theta(t)\in[0,2\pi)$ and $y=y(t)\in\mathbb{R}$, such that for any $t\in I_{\text{max}}$ (the maximal lifespan)
	\begin{align}
		\left\|u(t)-e^{i\theta}\lambda^{1/2}\phi_{\omega, 2\kappa_0\sqrt{\omega}}\left( \lambda(\ \cdot -y(t))\right)\right\|_{H^1}<\epsilon
	\end{align}
	where $\lambda=\lambda_{\omega, u}(t)$ is given by
	\begin{align*}
		\lambda^2=\frac{\|u(t)\|_{\dot{H}^1}^2+{b\over 3}\|u(t)\|_{L^6}^6}{\omega\|\phi_{1, 2\kappa_0}\|_{L^2}^2}.
	\end{align*}
\end{theorem}
We note that $\lambda\sim \|u(t)\|_{\dot{H}^1}$ by the Gagliardo–Nirenberg inequality, and exactly same as $\lambda=\frac{\|u(t)\|_{\dot{H}^1}}{\|\phi_{\omega, 2\kappa_0\sqrt{\omega}}\|_{\dot{H}^1}}$ when $b=0$. To prove Theorem \ref{MainTheorem}, we utilize the argument used in (\cite{Hayashi2021}, Theorem 1.7-(vi)), based on the variational argument in \cite{KwonWu2018}. By a further refinement of the proof of Theorem \ref{MainTheorem}, the following result can be obtained.
\begin{corollary}\label{MainCorollary}
	For $E(u_0)<(\kappa_0\sqrt{1+\kappa_0^2}-\kappa_0^2)\frac{(P(u_0))^2}{M(u_0)}$ instead of $E(u_0)\leq0$, Theorem \ref{MainTheorem} still holds true.
\end{corollary}
Due to a limitation of our argument, we cannot obtain any number better than $\kappa_0\sqrt{1+\kappa_0^2}-\kappa_0^2$. See Remark \ref{3remark}. From the experience of NLS, we expect that the sharp condition is $E(u_0)\leq \frac{1}{2}\frac{(P(u_0))^2}{M(u_0)}$. At least, we can know that Corollary \ref{MainCorollary} fails if $E(u_0)>\frac{1}{2}\frac{(P(u_0))^2}{M(u_0)}$. See Remark \ref{3remark2}

Indeed, Theorem \ref{MainTheorem} is a kind of classical proximity of the solution. Let us consider the following critical generalized KdV equation and nonlinear Schr\"odinger equation:
\begin{align}\label{gkdv}
	\partial_tu+\partial_{xxx}u+\partial_x(u^5)=0, \quad (t,x) \in \mathbb{R}\times\mathbb{R}, \tag{gKdV}
\end{align}
\begin{align}\label{NLS}
	i\partial_tu+\partial_{xx}u+|u|^4u=0, \quad (t,x) \in \mathbb{R}\times\mathbb{R}. \tag{NLS}
\end{align}
Both \eqref{gkdv} and \eqref{NLS} have a soliton which satisfies
\begin{align*}
	Q_{xx}+Q^5=Q, \quad Q\in H^1(\mathbb{R}).
\end{align*}
For any initial data $u_0$ such that slightly large mass condition $\|Q\|_{L^2}<\|u_0\|_{L^2}<\|Q\|_{L^2}+\delta$ with negative energy $E(u_0)<0$ and zero momentum $P(u_0)=0$ for NLS, the solutions $u$ of \eqref{gkdv} and \eqref{NLS} satisfy the following proximity
\begin{align}\label{proximityofQinNLS}
	\left\|Q-e^{i\theta}\lambda_0^{1/2}u(\lambda_0(\ \cdot -y))\right\|_{H^1}<\epsilon
\end{align}
with $\lambda_0=\frac{\|Q_x\|_{L^2}}{\|u_x\|_{L^2}}$ which is proved in \cite{Merle2001} for gKdV and \cite{MerleRaphael2005} for NLS. As can be seen in \cite{Merle2001} or \cite{MerleRaphael2005}, by using \eqref{proximityofQinNLS}, we can develope a modulation decomposition which is needed for a blow-up result. 

This modulation decomposition argument originated from \cite{MerleRaphael2001}. On the (g)DNLS, some instability results (\cite{GuoNingWu2020,Ning2020,FukayaHayashi2022} and the references therein) were proved by using the modulation decomposition similar to \cite{MerleRaphael2001}. However, the modulation decomposition used in the paper above is not sharp enough to prove the existence of blow-up solution. This is the reason why we need the proximity result such as \eqref{proximityofQinNLS} or Theorem \ref{MainTheorem}.

The global well-posedness of \eqref{dnls} with smallness of mass such as $M(u_0)<M(\phi_{1,2\kappa_0})$ was proved in \cite{Ozawa1996,Wu2013,Wu2015} for $b=0$, and \cite{Hayashi2021} for $b>0$.
Recently, when $b=0$, the global well-posedness of \eqref{DNLSoriginGauge} for large data was proved in many papers, for example \cite{PeliSaalShim2017,JenkinsLiuSulem2020,BahouriPerelman2022,HarKilNteVis2022arxiv}. Since these global results originate from integrability of \eqref{DNLSoriginGauge}, it is still an interesting question whether or not there is a blowup solution for $b>0$ and $M(u_0)>M(\phi_{1,2\kappa_0})$.

In section 2 we briefly review the variational characterization of the solitons. In section 3 we prove Theorem \ref{MainTheorem}.


\section{Variational characterization}
In this section, we consider the variational structure of the solitons. Unless otherwise noted, we assume $b\geq 0$. We define a scaling operator $f_\lambda(x)=\lambda^{1/2} f(\lambda x)$.
As in \cite{ColinOhta2006} and \cite{Hayashi2021}, there exists a two-parameter family of solitary waves
$$u_{\omega,c}(t,x)=e^{i\omega t}\phi_{\omega,c}(x-ct),$$
where $\phi_{\omega,c}(x)\in H^1(\mathbb{R})$ satisfies
\begin{align}\label{2solitonEquaition}
	-\phi^{\prime \prime}+\omega \phi+ic\phi^\prime-i|\phi|^{2}\phi^\prime-b|\phi|^4\phi=0.
\end{align}
The equaition \eqref{2solitonEquaition} is equivalent to $S_{\omega,c}^\prime(\phi)=0$ where
\begin{align*}
	S_{\omega,c}(\phi)
	&=E(\phi)+{\omega\over2}M(\phi)+{c\over2}P(\phi).
\end{align*}
Applying the following gauge transformation to $\phi_{\omega,c}$
\begin{align}\label{2solitonForm}
	\phi_{\omega,c}(x)
	=
	\Phi_{\omega,c}(x)\exp\left(i{c\over 2}x-{i\over 4}\int_{-\infty}^x |\Phi_{\omega,c}(y)|^{2}dy\right),
\end{align}
then $\Phi_{\omega,c}$ satisfies the following
\begin{align}
	\Phi_{\omega,c}^2(x)=
	\begin{cases}
		{2(4\omega-c^2) 
			\over 
			\sqrt{c^2+\gamma(4\omega-c^2)} \cosh(\sqrt{4\omega-c^2} x) -c }, 
		& \text{if } \omega>{c^2\over 4}
		\\
		{4c 
			\over 
			c^2x^2+\gamma},
		& \text{if } \omega={c^2\over 4} \text{ and } c>0,
	\end{cases} 
\end{align}
where $\gamma$ satisfies
\begin{align}\label{3gamma}
	\gamma=1+{16\over 3}b.
\end{align}
We also remark that $\Phi_{\omega,c}(x)$ is the $H^1$ solution of
$$-\Phi^{\prime \prime}+\bigg(\omega-{c^2\over 4}\bigg)\Phi+{c\over 2}|\Phi|^{2}\Phi-{3\over 16}\gamma|\Phi|^{4}\Phi=0.$$
We define some variational functionals as follow:
\begin{align*}
	K_{\omega,c}(\phi)
	&=\partial_\lambda S_{\omega,c}(\lambda\phi)|_{\lambda=1}
	\\
	&=\|\partial_x \phi\|_{L^2}^2+\omega\|\phi\|_{L^2}^2-c\text{Im}\int\overline{\phi}\partial_x \phi dx+\text{Im}\int |\phi|^{2}\overline{\phi}\partial_x \phi dx-b\|\phi\|_{L^6}^6
	\\
	&=-\|\partial_x \phi\|_{L^2}^2-{b\over 3}\|\phi\|_{L^6}^6+4E(\phi)+\omega M(\phi)+cP(\phi)
\end{align*}
We further define
\begin{align}
	N_1(\phi)&=-\text{Im}\int |\phi|^{2}\overline{\phi}\partial_x \phi dx \label{2N1definition}
	\\
	N_2(\phi)&=b\|\phi\|_{L^6}^6. \label{2N2definition}
\end{align}
We define functions spaces by
\begin{align*}
	f\in X_{\omega,c} \Longleftrightarrow
	\begin{cases}
		f\in H^1(\mathbb{R}) & 4\omega>c^2.
		\\
		e^{-i\frac{cx}{2}}f \in\dot{H}^1(\mathbb{R})\cap L^4(\mathbb{R}) & c=2\sqrt{\omega}. 
	\end{cases}
\end{align*}
Now, we consider the following minimization problem:
\begin{align}\label{2minimizationProb}
	\mu_{\omega,c}=\inf\{S_{\omega,c}(\phi): \phi\in X_{\omega,c}\setminus\{0\}, K_{\omega,c}(\phi)=0\}.
\end{align}

We define the sets $\mathcal{G}_{\omega,c}$ and $\mathcal{M}_{\omega,c}$ by
\begin{align*}
	\mathcal{G}_{\omega,c}
	&:=\{\phi \in X_{\omega,c}\setminus\{0\} : S_{\omega,c}^\prime(\phi)=0\},
	\\
	\mathcal{M}_{\omega,c}
	&:=\{\phi \in X_{\omega,c}\setminus\{0\} : S_{\omega,c}(\phi)=\mu_{\omega,c}, K_{\omega,c}(\phi)=0\}.
\end{align*}
$\mathcal{G}_{\omega,c}$ is the set of nontrivial solution of \eqref{2solitonEquaition}, and $\mathcal{M}_{\omega,c}$ is the set of all minimizers for \eqref{2minimizationProb}.
Now, we state the variational charicterization of $\phi_{\omega,c}$.
\begin{proposition}\label{2propMinimization}
	We have 
	\begin{align*}
		\mathcal{G}_{\omega,c}=
		\mathcal{M}_{\omega,c}=
		\{e^{i\theta}\phi_{\omega,c}(\cdot-y): (\theta, y) \in \mathbb{R}^2\}.
	\end{align*}
\end{proposition}
\begin{proposition}\label{prop uniqueness}
	Let $b\geq0$. Suppose that a sequence $\{u_n\}\subset H^1$ satisfies
	\begin{align}\label{prop uniqueness converge asumption}
		S_{\omega,c}(u_n)\to \mu_{\omega,c}, \quad K_{\omega,c}(u_n)\to 0.
	\end{align}
	For $(\omega, c)\in \mathbb{R}^2$ satisfing $-2\sqrt{\omega}<c\leq2\sqrt{\omega}$, there exist sequences $\{\theta_n\}, \{y_n\}\subset \mathbb{R}$ such that $e^{i\theta_n}u_n(\cdot-y_n)$ converges to $\phi_{\omega,c}$ strongly in $X_{\omega,c}$ up to subsequence.
\end{proposition}
We can find the proofs of results similar to the Proposition \ref{2propMinimization} and Proposition \ref{prop uniqueness} in many other papers, such as \cite{ColinOhta2006, KwonWu2018,Guo2018, Hayashi2021}. We use the linear profile decomposition. We remark that Proposition \ref{2propMinimization} implies that $S_{\omega,c}(\phi_{\omega,c})=\mu_{\omega,c}$. We also introduce the uniqueness of degenerate solitary waves.
\begin{proposition}\label{2energyequMomentum}
	Suppose $b\geq 0$ and $-2\sqrt{\omega}<c\leq2\sqrt{\omega}$. Then, we have
	\begin{align*}
		E(\phi_{\omega,c})=-\frac{c}{4}P(\phi_{\omega,c}).
	\end{align*}
	Moreover, there exists a unique $\kappa_0=\kappa_0(b)\in (0,1]$ with $\kappa_0(0)=1$ such that
	\begin{align*}
		E(\phi_{\omega,2\kappa_0\sqrt{\omega}})=P(\phi_{\omega,2\kappa_0\sqrt{\omega}})=0.
	\end{align*}
\end{proposition}
\begin{proof}
	See \cite{Ohta2014} or \cite{Hayashi2021}.
\end{proof}
Thanks to Proposition \ref{2energyequMomentum}, the zero energy soliton always has zero momentum. In addition, we have
\begin{align*}
	\phi_{\omega, 2\kappa_0\sqrt{\omega}}=(\phi_{1, 2\kappa_0})_{\sqrt{\omega}},
\end{align*}
and this means that the degenerate soliton is unique up to scaling and $M(\phi_{\omega, 2\kappa_0\sqrt{\omega}})=M(\phi_{1, 2\kappa_0})$ for any $\omega>0$. Finally, in this section, we introduce the theorem that motivated our main result.
\begin{theorem}\label{theorem minimum}
	If $E(\phi)\leq 0, P(\phi)\leq0$, and $M(\phi)=M(\phi_{1, 2\kappa_0})$, then $\phi=e^{i\theta}\phi_{\omega, 2\kappa_0\sqrt{\omega}}(\cdot-y)$ for some $\theta, y \in \mathbb{R}$, and $\omega>0$.
\end{theorem}
\begin{proof}
	See (\cite{Hayashi2021}, Theorem 1.7-(vi)).
\end{proof}
From above theorem, we can expect that if $M(\phi)$ is slightly larger then the critical $M(\phi_{1, 2\kappa_0})$ with negative energy and momentum, then $\phi$ is close to $\phi_{1, 2\kappa_0}$ in a sense.

\section{Proof of Theorem \ref{MainTheorem}}
We prove our main theorem in this section. Since $M(u), E(u), P(u)$ are conserved, it sufficies to show the following theorem to obtain Theorem \ref{MainTheorem}. We will combine the variational characterization in \cite{KwonWu2018} and the argument used in (\cite{Hayashi2021}, Theorem 1.7-(vi)). 
\begin{proposition}\label{theroem decompose}
	Let $b\geq 0$. For any $\epsilon>0$, there exists a $\delta=\delta(\epsilon)$ such that if $M^*=M(\phi_{1, 2\kappa_0})<M(u)<M^*+\delta$, $E(u)\leq 0$, and $P(u)\leq 0$, then we have
	\begin{align}
		\inf_{(\theta, y)\in\mathbb{R}^2}\left\|\phi_{\omega, 2\kappa_0\sqrt{\omega}}-e^{i\theta}u_{\lambda_0}(\cdot -y)\right\|_{H^1}<\epsilon
	\end{align}
	where $\lambda_{0}=\lambda_{0}(\omega, u)$ is given by
	\begin{align}\label{theorem uniquness lambda}
		\lambda_{0}^2={\omega\|\phi_{1, 2\kappa_0}\|_{L^2}^2\over \|\partial_x u\|_{L^2}^2+{b\over 3}\|u\|_{L^6}^6}
		={\|\partial_x \phi_{\omega, 2\kappa_0\sqrt{\omega}}\|_{L^2}^2+{b\over 3}\|\phi_{\omega, 2\kappa_0\sqrt{\omega}}\|_{L^6}^6\over \|\partial_x u\|_{L^2}^2+{b\over 3}\|u\|_{L^6}^6}.
	\end{align}
\end{proposition}
We note that
$$\|\partial_x\phi_{\omega, 2\kappa_0\sqrt{\omega}}\|_{L^2}^2=\omega\|\phi_{\omega, 2\kappa_0\sqrt{\omega}}\|_{L^2}^2-{b\over3}\|\phi_{\omega, 2\kappa_0\sqrt{\omega}}\|_{L^6}^6$$
from $E(\phi_{\omega, 2\kappa_0\sqrt{\omega}})=P(\phi_{\omega, 2\kappa_0\sqrt{\omega}})=0$ and $K(\phi_{\omega, 2\kappa_0\sqrt{\omega}})=0$.
\begin{proof}
	For simplicity of the notations, we use $S_{\omega}$, $K_{\omega}$, $\mu_{\omega}$ instead of $S_{\omega,2\kappa_0\sqrt{\omega}}$, $K_{\omega, 2\kappa_0\sqrt{\omega}}$, $\mu_{\omega, 2\kappa_0\sqrt{\omega}}$. We have
	$$K_{\omega}(u_\lambda)=\lambda^2(-\|\partial_x u\|_{L^2}^2-{b\over 3}\|u\|_{L^6}^6)+4\lambda^2E(u)+\omega M(u)+2\kappa_0\sqrt{\omega}\lambda P(u).$$
	Let $\lambda_0^2={\omega\|\phi_{\omega, 2\kappa_0\sqrt{\omega}}\|_{L^2}^2\over \|\partial_x u\|_{L^2}^2+{b\over 3}\|u\|_{L^6}^6}$, then we have
	\begin{align*}
		K_{\omega}(u_{\lambda_0})
		&=4\lambda_0^2E(u)+\omega (M(u)-M^*)+2\kappa_0\sqrt{\omega}\lambda_0 P(u),
		\\
		2S_{\omega}(u_{\lambda_0})
		&=2\lambda_0^2E(u)+\omega M(u)+2\kappa_0\sqrt{\omega}\lambda_0P(u).
	\end{align*}
	Our goal is proving that
	\begin{align}\label{3proofGoal}
		|K_{\omega}(u_{\lambda_0})|<\epsilon_0,\quad \text{and} \quad |2S_{\omega}(u_{\lambda_0})-\omega M^*|<\epsilon_0
	\end{align}
	for any sufficiently small $\epsilon_0$. We note that $\omega M^*=2\mu_\omega$. Then, by applying Proposition \ref{prop uniqueness}, we conclude the proof.
	
	We claim that for any $\epsilon>0$, there exists $\delta>0$ such that
	\begin{align}\label{theorem close proof 1}
		-\epsilon<\lambda_0^2 E(u)
	\end{align}
	if $M^*<M(u)<M^*+\delta$ and $E(u),P(u)\leq0$. We note that $\lambda_0\sim 1/ \|\partial_x u\|_{L^2}$.
	Suppose that the claim is not true. Then, there exists a sequence $\{u_n\}$ and $c>0$ such that 
	\begin{align}\label{theorem close proof 1-1}
		{E(u_n)\over \|\partial_x u_n\|_{L^2}^2} \to-c, \quad M(u_n)\downarrow M^*, \quad P(u_n)\leq 0.
	\end{align}
	Then, there exists $\lambda_n>0$ so that $M^*=M(\lambda_nu_n)=\lambda_n^2M(u_n)$, and this implies $\lambda_n\uparrow 1$ and $P(\lambda_nu_n)\leq 0$. Thus, we have $E(\lambda_nu_n)\geq 0$ by Theorem \ref{theorem minimum}. We have
	\begin{align}\label{theorem close proof 1-2}
		0\leq {E(\lambda_nu_n)\over \|\partial_x \lambda_nu_n\|_{L^2}^2}
		=
		{1\over 2}-{\lambda_n^2\over4}{N_1(u_n)\over\|\partial_x u_n\|_{L^2}^2}-{b\lambda_n^4\over 6}{N_2(u_n)\over \|\partial_x u_n\|_{L^2}^2},
	\end{align}
	where $N_1$ and $N_2$ are given by \eqref{2N1definition} and \eqref{2N2definition}.
	By \eqref{theorem close proof 1-1} and \eqref{theorem close proof 1-2}, we have that, for sufficiently large $n$,
	$$-{1\over4}{N_1(u_n)\over \|\partial_x u_n\|_{L^2}^2}-{b\over 6}{N_2(u_n)\over \|\partial_x u_n\|_{L^2}^2}\leq -{1\over2}-{2\over 3}\epsilon_1$$
	for sufficiently small $\epsilon_1<c$.
	Since $\lambda_n\uparrow 1$, there exists $N$ such that $1>\lambda_n^2>\lambda_n^4>1-\epsilon_1$ for $n>N$. If $N_1(u_n)\geq0$ for infiniely many $n$, we obtain
	$$-{\lambda_n^2\over4}{N_1(u_n)\over \|\partial_x u_n\|_{L^2}^2}-{b\lambda_n^4\over 6}{N_2(u_n)\over \|\partial_x u_n\|_{L^2}^2}\leq (1-\epsilon_1)\left(-{1\over2}-{2\over 3}\epsilon_1\right)=-{1\over2}-{1\over 6}\epsilon_1 +{2\over 3}\epsilon_1^2$$
	after passing a subsequence if necessary. Combining this and \eqref{theorem close proof 1-2},
	$0\leq -{1\over 6}\epsilon_1 +{2\over 3}\epsilon_1^2$ and this is contradiction for small $\epsilon_1>0$. If $N_1(u_n)\leq0$ for infiniely many $n$, then by \eqref{theorem close proof 1-1} and \eqref{theorem close proof 1-2}, along the subsequence, we have
	\begin{align}\label{theorem close proof 1-3}
		\lambda_n^4c-\epsilon_2\leq{1\over2}(1-\lambda_n^4)+{\lambda_n^4-\lambda_n^2\over 4}{N_1(u_n)\over \|\partial_x u_n\|_{L^2}^2}
	\end{align}
	for sufficiently small $\epsilon_2>0$ and large $n$.
	Taking $n\to \infty$, we have $\lambda_n \to 1$, and \eqref{theorem close proof 1-3} becomes, for sufficiently small $0<\epsilon_2<c$,
	\begin{align}\label{theorem close proof 1-4}
		{N_1(u_n)\over \|\partial_x u_n\|_{L^2}^2}\to -\infty.
	\end{align}
	However, H\"older inequality, Gagliardo–Nirenberg inequality, and $M(u_n)\leq 2 M^*$ for large $n$, we have
	$${|N_1(u_n)|\over \|\partial_x u_n\|_{L^2}^2}\lesssim {M^*}^2,$$
	and this is condtradiction with \eqref{theorem close proof 1-4}.
	Therefore, we have \eqref{theorem close proof 1}. 
	
	Now, again using the contradiction, we will show that
	\begin{align}\label{theorem close proof P}
		-\epsilon<\lambda_0P(u).
	\end{align}
	We use the idea of (\cite{Hayashi2021}, Theorem 1.7-(vi)) to prove \eqref{theorem close proof P}.
	Let $\widetilde{\lambda_n}$ be
	\begin{align}\label{3defTildeLambda}
		\widetilde{\lambda_n}^2=\frac{\lambda_0^2}{\omega \|\phi_{1, 2\kappa_0}\|_{L^2}^2}=\frac{1}{\|\partial_x u_n\|_{L^2}^2+{b\over 3}\|u_n\|_{L^6}^6}.
	\end{align} 
	If \eqref{theorem close proof P} is not true, we can find again a sequence $\{u_n\}$ and $c>0$ such that
	\begin{align*}
		M(u_n)\downarrow M^*, \quad \widetilde{\lambda_n}^2E(u_n) \uparrow 0, \quad
		\widetilde{\lambda_n}P(u_n)\to -c.
	\end{align*}
	We remark that the condition
	\begin{align*}
		\widetilde{\lambda_n}E(u_n) \uparrow 0
	\end{align*}
	comes from \eqref{theorem close proof 1}.
	Then, we have
	\begin{align*}
		K_{\omega}((u_n)_{\widetilde{\lambda_n}})&=-1+4\widetilde{\lambda_n}^2E(u_n)+\omega M(u_n)+2\kappa_0\sqrt{\omega}\widetilde{\lambda_n}P(u_n),
		\\
		2S_{\omega}((u_n)_{\widetilde{\lambda_n}})&=2\widetilde{\lambda_n}^2E(u_n)+\omega M(u_n)+2\kappa_0\sqrt{\omega}\widetilde{\lambda_n}P(u_n).
	\end{align*}
	We note that
	\begin{align*}
		2S_{\omega}((u_n)_{\widetilde{\lambda_n}})=K_{\omega}((u_n)_{\widetilde{\lambda_n}})+1-2\widetilde{\lambda_n}^2E(u_n).
	\end{align*}
	For any sufficiently small $\epsilon_3>0$, there exists $N$ such that if $n>N$, then 
	\begin{align*}
		-\epsilon_3<\widetilde{\lambda_n}^2E(u_n)\leq0, 
		\quad 0<M(u_n)-M^*<\epsilon_3, 
		\quad -\epsilon_3<\widetilde{\lambda_n}P(u_n)+c<\epsilon_3.
	\end{align*}
	We deduce that
	\begin{align*}
		-1-(4+2\kappa_0\sqrt{\omega})\epsilon_3+\omega M^*
		<
		K_{\omega}((u_n)_{\widetilde{\lambda_n}})
		<
		-1+(2\kappa_0\sqrt{\omega}+\omega)\epsilon_3+\omega M^*.
	\end{align*}
	We can check that $K_{\omega}((u_n)_{\widetilde{\lambda_n}})<0$ for sufficiently small $\omega$ and $K_{\omega}((u_n)_{\widetilde{\lambda_n}})>0$ for sufficiently large $\omega$. Therefore, there exists a $\omega_0>0$ such that $K_{\omega_0}((u_n)_{\widetilde{\lambda_n}})=0$. Moreover, since $K_{\omega}((u_n)_{\widetilde{\lambda_n}})$ is a quadratic function for $\sqrt{\omega}$, $\omega_0$ is unique on $\omega \in (0,\infty)$. We have
	\begin{align*}
		2S_{\omega_0}((u_n)_{\widetilde{\lambda_n}})=1-2\widetilde{\lambda_n}^2E(u_n)<1+2\epsilon_3.
	\end{align*}
	If $1+2\epsilon_3<\omega_0 M^*=2\mu_{\omega_0}$, then we obtain
	\begin{align*}
		K_{\omega_0}((u_n)_{\widetilde{\lambda_n}})=0\quad \text{and} \quad  S_{\omega_0}((u_n)_{\widetilde{\lambda_n}})<\mu_{\omega_0}.
	\end{align*}  
	This leads to a contradiction from the definition of $\mu_\omega$. Since  $K_{\omega}((u_n)_{\widetilde{\lambda_n}})$ is a quadratic function for $\sqrt{\omega}$, it sufficies to show that  $K_{\omega_1}((u_n)_{\widetilde{\lambda_n}})<0$ for $\omega_1=\frac{1+2\epsilon_3}{M^*}$. We have
	\begin{align*}
		K_{\omega_1}((u_n)_{\widetilde{\lambda_n}})
		&<-1+\frac{(1+2\epsilon_3)(M^*+\epsilon_3)}{M^*}+\frac{2\kappa_0\sqrt{1+2\epsilon_3}}{\sqrt{M^*}}(-c+\epsilon_3)
		\\
		&=\left(2+\frac{1+2\epsilon_3}{M^*}+\frac{2\kappa_0\sqrt{1+2\epsilon_3}}{\sqrt{M^*}}\right)\epsilon_3-\frac{2\kappa_0\sqrt{1+2\epsilon_3}}{\sqrt{M^*}}c.
	\end{align*}
	Taking $\epsilon_3$ sufficiently small, we obtain $K_{\omega_1}((u_n)_{\widetilde{\lambda_n}})<0$. Therefore, we can find $\omega_0$ such that 
	\begin{align*}
		K_{\omega_0}((u_n)_{\widetilde{\lambda_n}})=0 \quad \text{and} \quad S_{\omega_0}((u_n)_{\widetilde{\lambda_n}})<\mu_{\omega_0}.
	\end{align*}
	Then, we deduce the contradiction from the definition of $\mu_{\omega}$. Thus, we have \eqref{theorem close proof P}. Finally, we have
	\begin{align*}
		-(4+2\kappa_0\sqrt{\omega})\epsilon<K_{\omega}(u_{\lambda_0})<\omega\delta,
		\\
		-(1+\kappa_0\sqrt{\omega})\epsilon<S_{\omega}(u_{\lambda_0})-\mu_{\omega}<\frac{\omega\delta}{2}.
	\end{align*}
	Therefore, taking $\delta=\epsilon$, we deduce \eqref{3proofGoal}. Thanks to Proposition \ref{prop uniqueness}, we conclude the proof when $b>0$. For $b=0$, we still have only that
	\begin{align*}
		\inf_{(\theta, y)\in\mathbb{R}^2}
		\left\|e^{-i\frac{cx}{2}}\left(
		\phi_{\omega,2\kappa_0\sqrt{\omega}}
		-e^{i\theta}u_{\lambda_0}(\cdot -y)
		\right)\right\|_{\dot{H}^1\cap L^4}\to 0
	\end{align*}
	as $\delta \to 0$. We see that for any $\psi \in C^\infty_c$
	\begin{align}\label{3maintheoremProofWeakconverge}
		\int (\phi_{\omega, 2\kappa_0\sqrt{\omega}}-e^{i\theta_0}u_{\lambda_0}(\cdot-y_0))\psi dx \to 0\quad \text{as} \quad \delta \to 0,
	\end{align}
	for some $\theta_0=\theta_0(\delta) \in \mathbb{R}$ and $y_0=y_0(\delta) \in \mathbb{R}$, since we have $L^4$ convergence. Moreover, from $\phi_{\omega, 2\kappa_0\sqrt{\omega}}-e^{i\theta_0}u_{\lambda_0}(\cdot-y_0) \in L^2$ and \eqref{3maintheoremProofWeakconverge}, $\phi_{\omega, 2\kappa_0\sqrt{\omega}}-e^{i\theta_0}u_{\lambda_0}(\cdot-y_0)$ weakly $L^2$ converges to $0$ as $\delta\to 0$. From $\|u_{\lambda_0}\|_{L^2}^2\to \|\phi_{\omega, 2\kappa_0\sqrt{\omega}}\|_{L^2}^2=M^*$ as $\delta\to 0$, we conclude strong $L^2$ convergence, and we have
	\begin{align*}
		\inf_{(\theta, y)\in\mathbb{R}^2}\left\|\phi_{\omega, 2\kappa_0\sqrt{\omega}}-e^{i\theta}u_{\lambda_0}(\cdot -y)\right\|_{L^2}\to 0
	\end{align*}
	as $\delta\to 0$ when $b=0$.
\end{proof}
Since $M(u(t))$, $E(u(t))$, and $P(u(t))$ are conserved, we can deduce Theorem \ref{MainTheorem} by applying Proposition \ref{theroem decompose} for each $t\in I_{\text{max}}$. As the final step in this paper, we will prove Corollary \ref{MainCorollary}.
\begin{proof}[Proof of Corollary \ref{MainCorollary}] 
	We will prove this corollary in a similar way to the proof of Proposition \ref{theroem decompose}.
	Since we still have $P(u_0)\leq 0$, we have
	\begin{align}\label{Cor close proof 1}
		-\epsilon<\lambda_0^2 E(u).
	\end{align}
	Using the contradiction, we will show that
	\begin{align}\label{Cor close proof P}
		-\epsilon<\lambda_0P(u).
	\end{align}
	We redefine $\widetilde{\lambda_n}$ as 
	\begin{align*}
		\widetilde{\lambda_n}^2=\frac{\lambda_0^2M(u_n)}{\omega \|\phi_{1, 2\kappa_0}\|_{L^2}^2}=\frac{M(u_n)}{\|\partial_x u_n\|_{L^2}^2+{b\over 3}\|u_n\|_{L^6}^6}.
	\end{align*} 
	If \eqref{Cor close proof P} is not true, we can find a sequence $\{u_n\}$ and $c>0$ such that
	\begin{align*}
		M(u_n)\downarrow M^*, \quad
		\frac{\widetilde{\lambda_n}P(u_n)}{M(u_n)}\to -c, \quad \frac{\widetilde{\lambda_n}^2E(u_n)}{M(u_n)} \to ac^2,
	\end{align*}
	for some $0\leq a < \kappa_0\sqrt{1+\kappa_0^2}-\kappa_0^2$.
	We remark that the conditions $0\leq a$ and $a < \kappa_0\sqrt{1+\kappa_0^2}-\kappa_0^2$ comes from \eqref{Cor close proof 1} and initial data assumption respectively. We also note that
	\begin{align*}
		\frac{|\widetilde{\lambda_n}P(u_n)|}{M(u_n)}\leq\frac{|P(u_n)|}{\|u_n\|_{\dot{H}^1}M(u_n)^{1/2}}\leq 1
	\end{align*}
	by H\"older inequality.	This means that $0< c\leq 1$. Again, we have
	\begin{align*}
		\frac{K_{\omega}((u_n)_{\widetilde{\lambda_n}})}{M(u_n)}=-1+4\frac{\widetilde{\lambda_n}^2E(u_n)}{M(u_n)}+\omega +2\kappa_0\sqrt{\omega}\frac{\widetilde{\lambda_n}P(u_n)}{M(u_n)}.
	\end{align*}
	and
	\begin{align*}
		2\frac{S_{\omega}((u_n)_{\widetilde{\lambda_n}})}{M(u_n)}&=2\frac{\widetilde{\lambda_n}^2E(u_n)}{M(u_n)}+\omega +2\kappa_0\sqrt{\omega}\frac{\widetilde{\lambda_n}P(u_n)}{M(u_n)}.
	\end{align*}
	For ease of calculation, we take $n\to \infty$. We have
	\begin{align*}
		\lim_{n\to \infty}\frac{K_{\omega}((u_n)_{\widetilde{\lambda_n}})}{M(u_n)}
		&=K_{\infty,\omega}=-1+4ac^2+\omega -2\kappa_0\sqrt{\omega}c,
		\\
		\lim_{n\to \infty}2\frac{S_{\omega}((u_n)_{\widetilde{\lambda_n}})}{M(u_n)}
		&=2S_{\infty,\omega}=2ac^2+\omega -2\kappa_0\sqrt{\omega}c.
	\end{align*}
	Similarly with the proof of Proposition \ref{theroem decompose}, it sufficies to show that there exists a $\omega_0>0$ such that 
	\begin{align*}
		K_{\infty,\omega_0}<0\quad \text{and} \quad 2S_{\infty,\omega_0}=\omega_0
	\end{align*}
	for any $0<c\leq1$. 
	From $2S_{\infty,\omega_0}=\omega_0$, we have
	\begin{align*}
		\omega_0=\frac{a^2c^2}{\kappa_0^2},
	\end{align*}
	and
	\begin{align*}
		K_{\infty,\omega_0}=-1+2ac^2+\frac{a^2c^2}{\kappa_0^2}
		=\left(\frac{a}{\kappa_0}+\kappa_0 \right)^2c^2-c^2\kappa_0^2-1.
	\end{align*}
	Since $0\leq a < \kappa_0\sqrt{1+\kappa_0^2}-\kappa_0^2$, we have $K_{\infty,\omega_0}<0$ for any $0<c\leq1$. This implies that for sufficiently large $n$, there exists $\omega$ so that
	\begin{align*}
		K_{\omega}((u_n)_{\widetilde{\lambda_n}})=0, \quad \text{and} \quad 2S_{\omega}((u_n)_{\widetilde{\lambda_n}})<\mu_\omega.
	\end{align*}
	Therefore, we deduce a contradiction, and we have \eqref{Cor close proof P}. Moreover, we have $\lambda_0^2E(u)<(\kappa_0\sqrt{1+\kappa_0^2}-\kappa_0^2) (\lambda_0P(u))^2 \lesssim \epsilon ^2$, and this implies that we can apply Proposition \ref{2propMinimization}. Thus, the rest of the proof is identical to the proof of Proposition \ref{theroem decompose}.
\end{proof}
\begin{remark}\label{3remark}
	With our arguments, it seems to be impossible to extend the condition $E(u_0)<(\kappa_0\sqrt{1+\kappa_0^2}-\kappa_0^2)\frac{(P(u_0))^2}{M(u_0)}$. In the proof of Corollary \ref{MainCorollary}, we cannot lead the contradiction for some $0<c\leq 1$ when $a\geq \kappa_0\sqrt{1+\kappa_0^2}-\kappa_0^2$.
\end{remark}
\begin{remark}\label{3remark2}
	Let $u_0=e^{irx}\phi_{1,2\kappa_0}$. We can check that $E(u_0)=c(r)\frac{(P(u_0))^2}{M(u_0)}$, $M(u_0)=M^*$, and $P(u_0)<0$ with $c(r)\downarrow \frac{1}{2}$ as $r\to \infty$. However, $u_0$ is not solition $\phi_{1,2\kappa_0}$ up to the symmetries. This implies that the conclusion of Proposition \ref{theroem decompose} fails when $E(u_0)>\frac{1}{2}\frac{(P(u_0))^2}{M(u_0)}$.
\end{remark}

\noindent\textbf{Acknowledgement.}
The author appreciates Soonsik Kwon for helpful discussions and encouragement to this work. The author is partially supported by the National Research Foundation of Korea(NRF) grant funded by the Korea government(MSIT) (NRF-2019R1A5A1028324).

\bibliographystyle{plain} 
\bibliography{reference}

\end{document}